\documentclass[12pt]{article}
\usepackage{amssymb,amsthm,amsmath,latexsym}
\usepackage[]{amsthm,enumerate}
\usepackage{verbatim,bm}
\usepackage{url}
\usepackage{color}
\usepackage{comment}
\usepackage{lscape}
\usepackage{appendix}

\theoremstyle{plain}
\newtheorem{thm}{Theorem}

\newtheorem{lem}[thm]{Lemma}
\newtheorem{cor}[thm]{Corollary}

\newtheorem{rem}[thm]{Remark}

\newtheorem{definition}[thm]{Definition}

\title{Skew-regular quaternary Hadamard matrices \thanks{Dedicated to the memory of Robert Woodrow, in honour of his profound contributions to combinatorics and his years of dedicated service to our community.}}

\author{ Hadi Kharaghani\thanks{
Department of Mathematics and Computer Science, University of Lethbridge,  
Lethbridge, Alberta, T1K 3M4, Canada.
email: \texttt{kharaghani@uleth.ca}}, \and
Vlad Zaitsev\thanks{
Department of Mathematics and Computer Science, University of Lethbridge,  
Lethbridge, Alberta, T1K 3M4, Canada.
email: \texttt{vlad.zaitsev@uleth.ca}}}
\date{}
\begin{document}
\setlength{\parindent}{20pt}

\maketitle 
\begin{abstract}

This paper introduces and investigates a novel class of skew-regular quaternary Hadamard matrices. For every odd prime power $p$, we establish the existence of these matrices for all orders $1+p^2$, each characterized by a constant row sum of $1-pi$. Motivated by the growing importance of large-excess matrices in the maximum determinant problem and quantum nonlocality, we demonstrate that this class provides a robust framework for generating Hadamard matrices with exceptionally large excess.

\end{abstract}

\section{Introduction}

In a recent study \cite{skew-r}, the authors introduced the concept of skew-regular Hadamard matrices. A Hadamard matrix $H$ of order $4n^2$ is defined as skew-regular if it has row sums $\pm 2n$ and can be expressed in the form $H=I+Q$, where $Q$ is skew-symmetric ($Q^T = -Q$). While skew-regular Hadamard matrices of order $4n^2$ are conjectured to exist for all odd $n$, it has been established that no such matrices exist for orders $16n^2$.
In this paper, we generalize the concept of skew-regularity to quaternary Hadamard matrices and provide a novel construction for a class of these matrices of order $1+p^2$ for every odd prime power $p$. By leveraging the skewness of these structures, we generate an infinite class of quaternary Hadamard matrices. We further utilize their regularity properties to demonstrate that each newly generated matrix is also regular. Additionally, we show that every skew-regular quaternary Hadamard matrix in this class can be used to construct a skew-semi-regular quaternary Hadamard matrix of twice the order. A significant feature of these matrices is their application in constructing Hadamard matrices with large excess. The paper concludes by presenting a quaternary Hadamard matrix of order $50$ that is equivalent to a skew-semi-regular quaternary Hadamard matrix with row sums $\pm 5 \pm 5i$, as well as a matrix with a constant row sum of $1-7i$.

\section{Preliminaries} 

A \emph{quaternary Hadamard matrix} $H$ of order $n$ is a square matrix whose entries belong to the set $\{1, -1, i, -i\}$ and which satisfies the property $HH^* = nI_n$, where $H^*$ denotes the conjugate transpose of $H$. When the entries are restricted to the real set $\{1, -1\}$, the matrix is simply referred to as a {Hadamard matrix}. Two quaternary Hadamard matrices $H$ and $K$ are said to be (Hadamard) equivalent if there exist quaternary permutation matrices $P$ and $Q$ such that $H = P K Q$. A similar equivalence notion applies to Hadamard matrices.

\begin{definition}
The \emph{excess} of a Hadamard matrix $H = [H_{ij}]$ is defined as the sum of all its entries.
\end{definition}

The maximum excess of a Hadamard matrix is a unique invariant, meaning that Hadamard matrices with different maximum excesses are not Hadamard equivalent. For a detailed discussion, refer to \cite{ks, kou, SY}.

A (quaternary) Hadamard matrix is \emph{normalized} if all entries in its first row and first column are equal to $1$. Any (quaternary) Hadamard matrix can be made equivalent to a normalized one through simple transformations. Matrix entries of a square matrix $M$ are denoted by $M_{ij}$.

Regular quaternary Hadamard matrices were introduced by Kharaghani and Seberry in \cite{ks}.

A (quaternary) Hadamard matrix $H$ is considered \emph{regular} if all its row (or column) sums are equal to the same (complex) number. The order of a regular quaternary Hadamard matrix is necessarily a sum of two integer squares, see \cite{ks}. Regular Hadamard matrices possess constant row and column sums; the same holds in the quaternary case.

\begin{definition}
    A \emph{skew-regular} quaternary Hadamard matrix is a regular skew-type quaternary Hadamard matrix. 
\end{definition}

The row sum of a skew-regular Hadamard matrix of order $4n^2$ is $\pm 2n$.
The situation for quaternary Hadamard matrices is as follows. 

\begin{lem}
    Let $H$ be a skew-regular quaternary Hadamard matrix of order $n$. Then the row sum of $H$ is of the form $1+ki$.
\end{lem}
\begin{proof}
    Let $s=a+ib$ be the row sum of $H$ and let $j_n$ denote the all ones column vector of size $n$. Then \[Hj_n =sj_n,\]which means $j_n$ is an eigenvector of $H$, with eigenvalue $s$. Since $H$ is skew-type, \[H+H^* = 2I_n.\] For any eigenvalue $\lambda$ and corresponding eigenvector $v$ \begin{align*}
        v^*(H+H^*)v &= v^*(2I_n) v  \\
        v^*Hv + v^*H^*v &= 2v^*v \\ 
        v^* (\lambda v) + (Hv)^* v &=  2 v^*v\\
        \lambda v^* v + \overline{\lambda}v^*v &= 2v^*v \\ 
        \lambda + \overline{\lambda} &= 2.
    \end{align*}Since $s$ is an eigenvalue, the real part of the row sum $s$ must be $1$.
\end{proof}

For a vector $v = (v_1, \dots, v_n)$, the diagonal matrix $D$ with entries $D_{ii} = v_i$ for $i = 1, \dots, n$ is denoted by $\mathrm{diag}(v)$.

\section{A class of skew-regular quaternary Hadamard matrices}

\begin{definition}
Let $p$ be an odd prime power. For $x \in \mathbb{F}_p$ define the \textit{quadratic character function}
\[
\chi_p(x)=
\begin{cases}
0, & \text{if }x=0,\\[4pt]
1, & \text{if } x\in\mathbb{F}_p\setminus\{0\} \text{ is a square in }\mathbb{F}_p\setminus\{0\},\\[4pt]
-1, & \text{otherwise.}
\end{cases}
\]  
\end{definition}

Let $p$ be an odd prime power and $q=p^2$. Note that every nonzero element of $\mathbb{F}_p$ is a square in $\mathbb{F}_q$. Fix an element $t \in \mathbb{F}_q \setminus \mathbb{F}_p$. For every $c \in \mathbb{F}_p$, define the set \[C_c = c t + \mathbb{F}_p.\]Note that $C_0=\mathbb{F}_p$ and $\{C_c \mid c \in \mathbb{F}_p\}$ is a set of $p$ additive cosets that partition $\mathbb{F}_q$. Moreover, the sum over all field elements of the quadratic character function is zero.
\[\sum_{x\in\mathbb{F}_q}\chi_q(x)=0.\]

\begin{lem}\label{lem:qcsumsconst}
   Let $p$ be an odd prime power and $q=p^2$. For $n,m \in \mathbb{F}_p \setminus \{0\}$, we have 
   \[\sum_{x\in C_n} \chi_q(x) = \sum_{x \in C_m}\chi_q(x).\]
\end{lem}
\begin{proof}
   By the definition of the cosets, there exists a $k = nm^{-1} \in \mathbb{F}_p\setminus \{0\}$ such that $C_n = \{ kx \mid x \in C_m \}$. Since $k\in \mathbb{F}_p \setminus \{0\}$, we know $\chi_q(k)=1$. Thus, we evaluate the character sum as follows:
   \[
       \sum_{y \in C_n} \chi_q(y) = \sum_{x \in C_m} \chi_q(kx) = \chi_q(k) \sum_{x \in C_m} \chi_q(x) =  \sum_{x \in C_m} \chi_q(x).
   \]
\end{proof}

Our last step before presenting the main construction is to evaluate the quadratic character sums over cosets of $\mathbb{F}_p$.
\begin{lem}\label{qrlem}
Let $p$ be an odd prime power and $q=p^2$. For $t\in\mathbb{F}_{q}$ define
\[
S(t)=\sum_{x\in t+\mathbb{F}_{p}}\chi_q(x).
\]
Then
\[
S(t)=
\begin{cases}
p-1, & t\in\mathbb{F}_{p},\\[2pt]
-1,  & t\notin\mathbb{F}_{p}.
\end{cases}
\]
\end{lem}
\begin{proof}
 Since every nonzero element of $\mathbb{F}_p$ is a square in $\mathbb{F}_q$, for $t \in \mathbb{F}_p$
\[
S(t)=\sum_{x\in t+\mathbb{F}_p}\chi_q(x)=\sum_{x\in \mathbb{F}_p}\chi_q(x)
=p-1.
\]Next, suppose $t \in \mathbb{F}_q \setminus \mathbb{F}_p$. Since $\{C_c \mid c \in \mathbb{F}_p\}$ partition $\mathbb{F}_q$ we have
\[
\sum_{x\in\mathbb{F}_q}\chi_q(x)
=\sum_{x\in\mathbb{F}_p}\chi_q(x)+\sum_{x\in\mathbb{F}_q \setminus \mathbb{F}_p}\chi_q(x)=0.
\]By Lemma \ref{lem:qcsumsconst} character sums over $\{C_c \mid c \in \mathbb{F}_p\setminus 0\}$ are constant. Thus\[p-1=\sum_{x\in\mathbb{F}_p}\chi_q(x)=-\sum_{x\in\mathbb{F}_q \setminus \mathbb{F}_p}\chi_q(x) =-(p-1) \sum_{x \in C_1} \chi_q(x).\]Therefore \[\sum_{x\in t + \mathbb{F}_p} \chi_q(x)=\sum_{x \in C_1} \chi_q(x)=-1,\]which implies $S(t)=-1$.
\end{proof}

The main construction relies on manipulating the rows and columns of the standard Paley matrix, defined below.
\begin{thm}[\cite{paley1933}]\label{paley}
    Let $p$ be an odd prime power. Then there is a symmetric conference matrix \[C=(C_{x,y})_{x,y \in \{\infty\} \cup \mathbb{F}_q}\]of order $n=q+1$, defined by \begin{align*}
 C_{x,x} = 0&, \quad \forall\,x;\\
C_{\infty,x} = C_{x,\infty} = 1&,\quad \forall\,x\in\mathbb{F}_q;\\
C_{x,y} = \chi_q(x-y)&, \quad  x\neq y,\;x,y\in\mathbb{F}_q.
\end{align*}
\end{thm}

\begin{thm}\label{main}
    Let $p$ be an odd prime power. Then there exists a skew-regular quaternary Hadamard matrix of order $p^2+1$ with row sum $1-pi$.
\end{thm}
\begin{proof}
    Let $q=p^2$ and $C$ be the symmetric conference matrix of order $n=q+1$ defined in Theorem \ref{paley}. Since $C$ is symmetric, $H = I_n-iC$ forms a skew-type quaternary Hadamard matrix of order $n$. Index the first row and column of $H$ by the symbol $\infty$, and the remaining rows and columns by elements of $\mathbb{F}_q$ as in the construction of $C$. 

Next, we partition $\mathbb{F}_p\setminus \{0\}$ into two sets $P$ and $-P$. Let $\omega$ be a primitive element of $\mathbb{F}_p$. Define \[P = \left\{\omega^i \mid 0 \leq i < \frac{p-1}{2} \right\}, \text{ and } -P = \left\{-\omega^i \mid 0 \leq i < \frac{p-1}{2} \right\}.\] Then, using the additive cosets $\{C_c \mid c \in \mathbb{F}_p\}$, we partition the index set $\{\infty\}\cup\mathbb{F}_q$ into four parts:
\[
\{\infty\},\quad \mathbb{F}_p= C_0, \quad \mathcal{H}_1=\bigcup_{c \in P}C_c, \quad \mathcal{H}_2=\bigcup_{c \in -P}C_c.
\]For the sake of cleaner notation, let \[\mathcal{C}=\bigcup_{c \in \mathbb{F}_p \setminus \{0\}} C_c.\]
Next, define the vector $v=(v_\alpha)$ by
\[
v_\infty = 1,
\qquad
v_x = 
\begin{cases}
1, & \text{if }x\in\mathbb{F}_p,\\
-i, & \text{if }x\in \mathcal{H}_1, \text{ and}\\
+i, & \text{if }x\in \mathcal{H}_2.
\end{cases}
\]
Finally, define
\[
M =\mathrm{diag}(v), \text{ and }
S = MHM^*.
\]
We will show that $S$ is the desired skew-regular quaternary Hadamard matrix. 

First, $S$ is skew-type since $H$ is skew-type. Next, we must show $S$ is regular. First, notice that the $(\alpha,\beta)$ entry of $S$ is given by
\[
S_{\alpha,\beta}
= v_\alpha H_{\alpha,\beta}\overline{v_\beta}.
\]Then, define \[
 R_\alpha =\sum_{\beta\in \mathbb{F}_q\cup \{\infty\}}H_{\alpha,\beta}\overline{v_\beta}.
\]Thus, the row sum of the row indexed by $\alpha$ is given by
\[
\sum_{\beta\in \mathbb{F}_q\cup \{\infty\}}S_{\alpha,\beta}
= \sum_{\beta\in \mathbb{F}_q\cup \{\infty\}}v_\alpha H_{\alpha,\beta}\overline{v_\beta} = v_\alpha R_\alpha.\]
It suffices to show that for every $\alpha\in \mathbb{F}_q\cup \{\infty\}$, $v_\alpha R_\alpha=1-pi$. We consider four cases for the row–index $\alpha$.

\noindent\textbf{Case 1: $\alpha=\infty$}. By definition
\[
H_{\infty,\infty}=1,
\qquad
H_{\infty,x}=-i\;\;(\forall\,x\in\mathbb{F}_q).
\]
Hence
\[
v_\alpha R_\alpha
= 1\overline{v_\infty}
+\sum_{x\in\mathbb{F}_q}(-i)\;\overline{v_x}
= 1
+ \Bigl[p\cdot (-i)
+\tfrac{q-p}2\cdot(-1)
+\tfrac{q-p}2\cdot(1)\Bigr]
= 1-pi.
\]

\noindent\textbf{Case 2: $\alpha\in\mathbb{F}_p$}.
Split the sum $R_\alpha=\sum_{\beta}H_{\alpha,\beta}\,\overline{v_\beta}$ into four parts:
\begin{enumerate}
  \item \(\beta=\infty\): We have \[v_\alpha H_{\alpha,\infty}\overline{v_\infty}  = -i. \]
  \item \(\beta=\alpha\): We have \[v_\alpha H_{\alpha,\alpha}\overline{v_\alpha} = 1. \]
  \item \(\beta\in\mathbb{F}_p\setminus\{\alpha\}\): since $\alpha,\beta \in \mathbb{F}_p$, we have $\chi_q(\alpha-\beta)=1$, which gives \[v_\alpha \cdot \sum_{\beta\in\mathbb{F}_p\setminus\{\alpha\}}H_{\alpha,\beta} = 1 \cdot \sum_{\beta\in\mathbb{F}_p\setminus\{\alpha\}}\chi_q(\alpha-\beta) \cdot (-i) =-i(p-1). \]
  \item \(\beta\in\mathbb{F}_q\setminus\mathbb{F}_p\): Using Lemma~\ref{qrlem} \begin{align*}
  v_\alpha \cdot \sum_{\beta \in\mathbb{F}_q\setminus\mathbb{F}_p}(-i)\cdot\chi_q(\alpha - \beta)\overline{v_\beta} 
  &=-i\cdot \sum_{\beta \in \mathcal{C}} \chi_q(\alpha - \beta)\overline{v_\beta} \\ 
  &= -i\cdot\sum_{\beta \in \mathcal{H}_1} \chi_q(\alpha - \beta)\cdot i \; -i\cdot  \sum_{\beta \in\mathcal{H}_2} \chi_q(\alpha - \beta)\cdot(-i) \\ 
  &= -\frac{p-1}{2}  + \frac{p-1}{2} \\ 
  &= 0.\end{align*}
\end{enumerate}
Adding the four parts gives
\[
v_\alpha R_\alpha
=-i + 1 - i(p-1) + 0
= 1-pi.
\]

\noindent\textbf{Case 3: $\alpha\in \mathcal{H}_1$}
We split $R_\alpha$ into five parts. Assume $\alpha \in C_k$ for $k \in P$.
\begin{enumerate}
  \item \(\beta=\infty\): We have \[v_\alpha H_{\alpha,\infty}\overline{v_\infty} = -i \cdot -i \cdot 1 = -1. \]
  \item \(\beta=\alpha\): We have \[v_\alpha H_{\alpha,\alpha}\overline{v_\alpha} = -i \cdot 1 \cdot i = 1. \]
  \item \(\beta\in\mathbb{F}_p\): Since $\alpha \notin \mathbb{F}_p$, we apply Lemma~\ref{qrlem} to obtain \[v_\alpha \sum_{\beta\in\mathbb{F}_p}-i\cdot\chi_q(\alpha-\beta) \overline{v_\beta} = -1 \cdot \sum_{\beta \in \mathbb{F}_p}\chi_q(\alpha-\beta)\cdot 1 = -1 \cdot (-1) = 1.\]
  \item \(\beta\in C_k\): In this case, $\alpha$ and $\beta$ are in the same coset, so $\alpha-\beta \in \mathbb{F}_p$. Hence \[v_\alpha \sum_{\beta\in C_k \setminus\{\alpha\}}-i\cdot\chi_q(\alpha-\beta)\overline{v_\beta} = -1\cdot \sum_{\beta\in C_k\setminus\{\alpha\}}\chi_q(\alpha-\beta)\cdot i =  -i\cdot\sum_{x \in \mathbb{F}_p \setminus \{0\}}\chi_q(x)= -i\cdot(p-1).\]
    \item \(\beta\in \mathbb{F}_q \setminus (\mathbb{F}_p\cup C_k)\): We have \begin{align*} v_\alpha \sum_{\beta\in \mathbb{F}_q \setminus (\mathbb{F}_p\cup C_k)}-i\cdot \chi_q(\alpha-\beta)\overline{v_\beta} 
    &= -\!\!\!\! \sum_{\beta \in \mathcal{H}_1 \setminus C_k}\big[\chi_q(\alpha-\beta)\cdot i\big]  - \sum_{\beta \in \mathcal{H}_2}\big[\chi_q(\alpha-\beta)\cdot (-i)\big] \\ 
        &= -i\cdot\bigg[-\frac{p-1}{2}+1\bigg]+i\cdot\bigg[-\frac{p-1}{2}\bigg] \\ 
        &= -i.\end{align*}
\end{enumerate}

Adding everything,
\[
v_\alpha R_\alpha
=-1 +1+ 1 -i\cdot (p - 1) - i = 1-pi.
\]

\noindent\textbf{Case 4: $\alpha\in \mathcal{H}_2$}
We split $R_\alpha$ into five parts as above. Parts (1-4) contribute $1$, $1$, $-1$ and $-i\cdot(p-1)$, respectively. The final part also follows closely. Assume $\alpha \in C_k$ for some $k \in -P.$
\begin{enumerate}  \setcounter{enumi}{4}
    \item \(\beta\in \mathbb{F}_q \setminus (\mathbb{F}_p\cup C_k)\): We have \begin{align*} v_\alpha \sum_{\beta\in \mathbb{F}_q \setminus (\mathbb{F}_p\cup C_k)}-i\cdot\chi_q(\alpha-\beta)\overline{v_\beta} 
        &= i \cdot \sum_{\beta \in \mathcal{H}_1 }\chi_q(\alpha-\beta) \; +i\cdot \sum_{\beta \in \mathcal{H}_2 \setminus C_k}\big[\chi_q(\alpha-\beta)\cdot (-1)\big] \\ 
        &= i\cdot\bigg[-\frac{p-1}{2}\bigg]-i\cdot\bigg[-\frac{p-1}{2}+1\bigg] \\ 
        &= -i.\end{align*}
\end{enumerate}

Adding everything,
\[
v_\alpha R_\alpha
=1 + 1 - 1 -i\cdot(p - 1) -i = 1 -pi.
\]
Therefore, for any $\alpha \in \mathbb{F}_q\cup \{\infty\}$ the row sum $v_{\alpha}R_\alpha=1-pi$. Hence, $S$ is skew-regular.
\end{proof}

An example of the construction is provided in Appendix \ref{app}.

\begin{lem}\label{plus}
Let $H$ be a quaternary Hadamard matrix with row sum $a+ib$. Then there is a quaternary Hadamard matrix with row sum $a-ib$.
\end{lem}
\begin{proof}
Write $H=A+iB$, where $A$ and $B$ are $(0,1,-1)$-matrices. It can be seen that the quaternary Hadamard matrix $K=A-iB$ has row sum $a-ib$.
\end{proof}

\begin{cor}\label{1picor}
There is a skew-regular quaternary Hadamard matrix of order $1+q^2$ with row sum $1+qi$ for each odd prime power $q$.
\end{cor}
\begin{proof}
Note that if the matrix $H$ in Lemma \ref{plus} is skew-type, then the matrix $K$ is skew-type too. The result follows from Theorem \ref{main} and Lemma \ref{plus}.
\end{proof}

\section{Some applications}

Skew-type (quaternary) Hadamard matrices provide a robust framework for the construction of larger matrix classes \cite{SY}. A significant advantage of skew-regular (quaternary) Hadamard matrices is that the regularity property is preserved under these constructions, yielding extensive families of regular quaternary Hadamard matrices. In this section, we investigate several applications of skew-regular quaternary Hadamard matrices, paralleling established results in the Hadamard case \cite{skew-r}.
While the new results presented here are novel, part of the methodologies employed are standard. Accordingly, our presentation is concise, with a primary focus on highlighting new findings.

\subsection{Quaternary orthogonal designs}
There are several methods for generalizing quaternary Hadamard matrices. For this paper, we focus on quaternary orthogonal designs. 
\begin{definition}
Let $x_1,x_2,\dots,x_k$ be real indeterminates. Then a \emph{quaternary orthogonal design} is a square matrix $D$ of order $n$, with entries from the set $\{0,\pm \epsilon_1x_1,\dots,\pm \epsilon_kx_k\}$ such that $DD^T = \sigma I_n$, where $\sigma = \sum_{\ell=1}^k s_\ell x_\ell^2$ and $\epsilon_i \in \{1, i\}$. We say that $D$ is a COD$(n;s_1,s_2,\dots,s_k)$ of type $(s_1,s_2,\dots,s_k)$ in variables $x_1,x_2,\dots,x_k$. 
\end{definition} 
By replacing each indeterminate with unity, one obtains a quaternary weighing matrix of weight $\sum_{\ell=1}^k s_\ell$; and if $\sum_{\ell=1}^k s_\ell = n$, then one has a quaternary Hadamard matrix, \cite{SY}. 

\begin{lem}\label{recur1}
There is a COD$(1+p^2;1,p^2)$ for each odd prime power $p$.
\end{lem}
\begin{proof}
Let $H=I+Q$ be the skew-regular quaternary Hadamard matrix of order $1+p^2$ constructed in Theorem \ref{main}. Then $aI+bQ$ is a COD$(1+p^2;1,p^2)$ with the constant row sum $a-pbi$.
\end{proof}

\begin{thm}\label{main3}
There is a COD$(2+2p^2;1,1,p^2,p^2)$ for each odd prime power $p$.
\end{thm}
\begin{proof}
Let $H=I+Q$ be the skew-regular quaternary Hadamard matrix of order $1+p^2$ constructed in Theorem \ref{main}. Then \[ I\otimes A+Q\otimes B, \] where, \[ A=\begin{bmatrix}a & -b\\b & a \end{bmatrix}, B=\begin{bmatrix}c & d\\d &-c \end{bmatrix}.\]
\end{proof}

\begin{definition} Let $H$ be a skew-type quaternary Hadamard matrix. Then $H$ can be written as $$H=\left(\begin{array}{c|c} 1& e\\\hline -e^T& I+Q\end{array}\right),$$ where $Q^*=-Q$. The matrix $I+Q$ is called the \emph{skew-core} of $H$.
\end{definition}
Note that the row and column sums of the matrix $Q$ are zero, and the inner product of any two distinct rows of $I+Q$ is minus one.
\subsection{A recursive construction}
\begin{thm}\label{recur2}
There is a $$COD\left((1+p^2)p^{2k};p^{2k},p^{2k+2}\right)$$ for each odd prime power $p$ and non-negative integer $k$.
\end{thm}
\begin{proof}
Let $I+Q$ be the skew-core of the quaternary Hadamard matrix in Theorem \ref{main}. Assuming that the COD$(1+p^2;1,p^2)$ constructed in Lemma \ref{recur1} is in variables $a,b$, with $a$ repeated one time. Recursively changing $b$ to $aI+bQ$ and $a$ to $bJ_{p^2}$ the result follows, see also \cite{kbtz}.
\end{proof}
\begin{rem}
Note that the regularity of the matrix is not used in the proof of Theorem \ref{recur2}, and the quaternary Hadamard matrices obtained from the orthogonal designs constructed in Theorem \ref{recur2} are not necessarily of skew-type.  \end{rem}

\begin{lem}\label{regular1}
For every odd prime power $p$, there is a regular quaternary Hadamard matrix of order $p^{2k}(1+p^2)$ with row sum $p^{2k}-p(p^{2k}i)$ for each odd prime $p$ and non-negative integer $k$.
\end{lem}
\begin{proof}
Applying Theorem \ref{recur2} and tracing the changes in row sums as $k$ increases in the recursive process, for $a=b=1$, is as follows:
\begin{enumerate}
    \item[(i)] For $k=0$, the order is $1+p^2$ and the row sum is $1-pi$,
\item[(ii)] for $k=1$, the order is $p^2(1+p^2)$ and the row sum is $p^2-pi$, and
\item[(iii)] for $k=2$, the order is $p^4(1+p^2)$ and the row sum is $p^2 - p^3i$.
\end{enumerate}
Inductively,
\begin{enumerate}
 \item[(iv)] for even $k\geq 0$, the order is $p^{2k}(1+p^2)$ and the row sum is $p^{k}-(p^{k+1})i$, and
  \item[(v)] for odd $k\geq 0$, the order is $p^{2k}(1+p^2)$ and the row sum is $p^{k+1}-(p^{k})i$.
\end{enumerate}
\end{proof}

\subsection{Skew biregular quaternary Hadamard matrices}

A \textit{biregular} quaternary Hadamard matrix takes two values for its row sums.  The order of a skew-type quaternary Hadamard matrix can be doubled.
\begin{lem}\label{doubling}
Let $H$ be a skew-regular quaternary Hadamard matrix of order $n=a^2+b^2$ with the row sum $a+bi$. Then the row sums of the skew-type matrix \[\begin{bmatrix} H & iH\\iH^* & H^* \end{bmatrix}\] belong to $\{a-b+(a+b)i, a+b+(a-b)i\}$, which is not regular in general.
 \end{lem}

\begin{definition}
 A \textit{skew-biregular} quaternary Hadamard matrix is a biregular skew-type quaternary Hadamard matrix.
 \end{definition}

  \begin{rem}
     For $p=7$, Theorem~\ref{main} produces a skew-regular quaternary Hadamard matrix of order 50, with the row sum $1-7i$. This particular matrix is also shown to be equivalent to a skew-biregular quaternary Hadamard matrix with row sums in $\{5+5i,-5-5i\}$ in Appendix~\ref{app2}.
 \end{rem} 
 
The following result follows directly from Theorem \ref{main} and Lemma \ref{doubling}.
\begin{cor} There is a skew-biregular quaternary Hadamard matrix of order $2+2p^2$ for each odd prime power $p$ with row sums in $\{1+p+(1-p)i,1-p+(1+p)i\}$.\end{cor}

\begin{lem} There is a regular quaternary Hadamard matrix of order $2+2p^2$ for each odd prime power $p$.\end{lem}
\begin{proof}
    Let $H$ be a skew-regular quaternary Hadamard matrix of order $1+p^2$ with the row sum $1-pi$ constructed using Theorem \ref{main}. Then the matrix \[ K=\begin{bmatrix} H & iH\\iH & H \end{bmatrix}\]is a regular quaternary Hadamard matrix of order $2+2p^2$ with row sums $\{1+p+(1-p)i\}$. Moreover, applying Corollary \ref{1picor} we can take $H$ to have row sums $1+pi$ to form a regular quaternary Hadamard matrix with row sums $\{1-p+(1+p)i\}$.
\end{proof}

\begin{lem}\label{biregular}
For every odd prime power $p$, there is a biregular quaternary Hadamard matrix of order $(2+2p^2)(1+2p^2)^k$  for each non-negative integer $k$.
\end{lem}
\begin{proof}
Let $D$ be the COD$(2+2p^2;1,2p^2+1)$ constructed in Lemma \ref{main3} on variables $a$ and $b$ by taking $b=c=d$. Let $q=1+2p^2$. The row sums of $D$ belongs to $\{a+b,a-b+2pai\}$. Let $Q_q$ be the skew-core of the skew-type quaternary Hadamard matrix of order $q$ constructed in Lemma \ref{doubling}. By recursively changing $b$ to $aI+bQ_q$ and $a$ to $bJ_{q}$, we get a biregular COD$(q^k(q+1);q^k,q^{k+1})$, say $D_k$, for each positive integer $k$, with row sums as follows.
 \begin{itemize}
 \item Row sums of $D$ belongs to $\{a+b, a-b+2pai\}$.
 \item Row sums of $D_{2m}$ belongs to $\{q^m(a-b+2pbi),q^m(a+b)\}$.
 \item Row sums of $D_{2m+1}$ belongs to $\{q^m(qb-a+2pai),q^m(qb+a)\}$.
 \end{itemize}
\end{proof}

 \subsection{An application to the excess problem}
 
 As demonstrated in the cases of orders $32$, $40$, and $44$ \cite{order32, wolfhadi, kou, kou40}, the substantial variation in excess across equivalence classes poses a major obstacle to determining maximum excess values. Although the theoretical upper bound for the excess of a Hadamard matrix of order $4n$ is $4n\sqrt{4n}$, this bound is rarely attainable in practice. Currently, sharper and more tractable upper bounds are known for certain structured orders, specifically $4m^2$, $(2m+1)^2 + 3$, and $(2m+1)^2 - 1$. For these orders, several infinite families of Hadamard matrices have been constructed that achieve their respective bounds exactly \cite{koji-sho,hadi,ksexcess, kou}. For most other orders, however, determining the maximum excess remains challenging, as excess values can vary widely across inequivalent matrices.

Despite these difficulties, interest in matrices with large excess continues to grow, driven both by their role in generating large determinants \cite{max-det} and their applications in quantum nonlocality \cite{karol}. 

Recent progress has further shown that quaternary Hadamard matrices can be constructed to lead to the maximum excess in real cases  \cite{ks}.
In this section, we show that skew-regular quaternary Hadamard matrices provide an effective framework for constructing Hadamard matrices with exceptionally high excess. We begin by recalling several known results.

\begin{lem}\label{kron}
Let $H$ be a matrix of order $n$ and row sum $a+ib$, $K$ a matrix of order $m$ and row sum $c+id$, then the  matrix $H\otimes K$ of order $nm$ has the row sum $(a+bi)(c+di)=ac-bd+i(ad+bc)$.\end{lem}
The following lemma allows us to convert quaternary Hadamard matrices to Hadamard matrices.
\begin{lem}[\cite{cohn}]
Let $H$ be a quaternary Hadamard matrix of order $n$. Write $H=A+iB$, where $A, B$ $(0,\pm 1)$-matrices of order $n$. Then 
\[ K=A\otimes C + B\otimes D\] where \[C=\begin{bmatrix} - & 1\\1 & 1\end{bmatrix} \text{ and } 
D=\begin{bmatrix} 1 & 1\\1 & -\end{bmatrix}\] is a Hadamard matrix of order $2n$. 
\end{lem}

We now present a lemma necessary for the application of skew-regular matrices. While the proof is elementary, it is omitted here due to its length.
 
\begin{lem}\label{min}
Let $p,q$, $1\le p\le q$ be two integers and $p \neq 2$. Then
\begin{equation*}
    2(1+p^2)(1+q^2) \left( \sqrt{2(1+p^2)(1+q^2)} - (pq+p+q-1) \right) \ge 8(1+q^2)(\sqrt{1+q^2}-q)
\end{equation*}
Furthermore, equality is attained if and only if $p=1$. \end{lem}

\begin{rem}\label{r1}
To construct Hadamard matrices with large excess using the Kronecker product, it is necessary to utilize quaternary Hadamard matrices with row sums $a+ib$ where the difference $|a-b|$ is as small as possible; a requirement formally established in Theorem \ref{main-excess}. Let $H_p$ and $H_q$ be skew-regular quaternary Hadamard matrices corresponding to prime powers $p$ and $q$, and let $H_{pq} = H_p \otimes H_q$ denote their Kronecker product. Corresponding to $p=1$, the smallest skew-regular quaternary Hadamard matrix is 
\[ M=\begin{bmatrix} 1 & i\\ i &1\end{bmatrix}. \] \end{rem}

\begin{thm}\label{main-excess}
For $p=1$ or odd prime powers $1 \le p < q$, the largest excess of the Hadamard matrix derived from $H_{pq}$ is attained in the specific case where $p=1$.
\end{thm}
\begin{proof}
The quaternary Hadamard matrix $H_{pq}$ is of order $n=(1+p^2)(1+q^2)$ with the row sum $a+ib$, where $a=1-pq$ and $b=-(p+q)$. Converting the quaternary Hadamard matrix $H_{pq}$ to a Hadamard matrix, we have a matrix of order $2n$ with the row sum in $\{2-2pq, -2(p+q)\}$. By negating the matrix, we have a matrix $\mathbb{H}$ of order $2n$ and excess $2n(pq+p+q-1)$. An upper bound for the excess of $\mathbb{H}$ is $2n\sqrt{2n}$. Applying Lemma \ref{min}, we have 
\[ 2n\left(\sqrt{2n} - (pq+p+q-1)\right) \ge 8(1+q^2)(\sqrt {1+q^2}-q),\]
with equality attained if $p=1$, thereby minimizing the difference between the excess of $\mathbb{H}$ and its theoretical upper bound. Note that for $p=1$, $a=1-q$ and $b=-1-q$, so $|a-b|=|1-q+1+q|=2$, which is as small as possible for $p=1$ or odd prime powers $1 \le p < q$.
\end{proof}

In reference to Theorem \ref{main-excess}, an application of skew-regular Hadamard matrices of order $1+p^2$, $p$ an odd prime power, constructed in Theorem \ref{main}, we introduce a new infinite class of Hadamard matrices with maximum excess. We make use of the following lemma, see~\cite{GS}.
\begin{lem}\label{exboundwe}
    Let $W$ be a weighing matrix $W(n,k^2)$. Then the excess of $W$ is at most $nk$, and the excess is equal to $nk$ if and only if all row sums of the matrix are equal to $k$.
\end{lem}

\begin{thm}\label{main2}
For each odd prime power $p$, there is a Hadamard matrix of order $4+4p^2$ with the largest excess $8p(1+p^2)$. \end{thm}
\begin{proof}
Let $H=I+Q$ be the skew-regular quaternary matrix constructed in Theorem \ref{main}. Consider the three quaternary weighing matrices \[Q_1=M\otimes H= \begin{bmatrix} H & iH\\iH & H\end{bmatrix},\] \[Q_2=\begin{bmatrix} Q & iQ\\iQ & Q\end{bmatrix},\] and 
\[Q_3=\begin{bmatrix} I & iI\\iI & I\end{bmatrix}.\] Writing $Q_j=A_j+iB_j$, $j=1,2,3$, where $A_j$ and $B_j$ are $(0,\pm 1)$-matrices. Converting the three matrices to (real) weighing matrices, we have for $j=1,2,3,$
$$W_j=A_j\otimes \begin{bmatrix} 1 & 1\\1 &-\end{bmatrix} + B_j\otimes \begin{bmatrix} - & 1\\1 &1\end{bmatrix}.$$  Then $W_1$ is a weighing matrix $W(4+4p^2,4+4p^2)$ (i.e., a Hadamard matrix), $W_2$ is a weighing matrix $W(4+4p^2,4p^2)$, and $W_3$ is a weighing matrix\\ $W(4+4p^2,4)$, where $W_1=W_2+W_3$. Then define $S_{j,k}$, for $j=1,2,3$, and $k=1,2,\ldots,4+4p^2,$ to be the $k$-th row sum of the $j$-th weighing matrix. From the construction, we know
\begin{itemize}
\item $S_{1,k}$ alternates between $2+2p$ and $2-2p$, respectively,
\item $S_{2,k}$ alternates between $2p$ and $-2p$, respectively, and
\item $S_{3,k}=2$, for $k=1,2,\cdots,4+4p^2$.
\end{itemize}
Negating each row with negative row sums in $W_1$ increases the row sums of $W_2$ by $4p$, and decreases the row sums of $W_3$ by $4$, thus increasing the row sums of $W_1$ by $4p-4$ for each negation. After negating all the negative row sums in $W_1$, the total sum of $W_3$ is zero, and by Lemma~\ref{exboundwe}, $W_2$ achieves its maximum excess of $2p(4+4p^2)$. Ignoring the zero entries, after negations of the rows of $W_1$, we can assume that $W_3$ consists of blocks of the form
\[B=\begin{pmatrix} 1 & 1 &- &1\\- & 1 & - &-\\- & 1 &1 & 1\\-  &- &- &1 \end{pmatrix}.\]
 Since both the row sums and column sums of $W_2$ are all $2p$, negating columns and rows of $B$ does not increase the total sum of $W_1$. Since no row and column negations can increase the excess, it follows that the maximum excess of the Hadamard matrix $W_1$ of order $4+4p^2$ is $8p(1+p^2)$.
\end{proof}
\begin{rem}
For $p=3$, Theorem \ref{main2} provides a Hadamard matrix of order $40$ with the maximum excess of $240$. A Hadamard matrix of order $40$ with the largest excess of $240$ was constructed in \cite{wolfhadi}. The largest excess achieved by Hadamard matrices of order $40$ is shown in \cite{kou40} to be $244$.  \end{rem}
\begin{rem} An infinite class of Hadamard matrices of order $n^2+3$ is known achieving the maximum excess of $n(n^2+3)$ \cite{koji-sho,koukouvinos1990}. The Hadamard matrices shown in Theorem \ref{main2} are of order $(2p)^2+4$, achieving a maximum excess of $(2p)(4+4p^2)$. \end{rem}
\subsection{A second class of matrices with large excess}
While the class of regular quaternary Hadamard matrices of order $n$ introduced in the preceding section satisfies $|a-b|=2$ and possesses an excess of $n\sqrt{n-4}$, we now consider matrices of order $m$ where $|a-b|=4$. As we shall demonstrate, these matrices yield an excess of $m\sqrt{m-16}$. 
 \begin{lem}\label{twin}
Let $p$ and $q$ be consecutive odd integers. Then the following identity holds:
\[
p^2 + q^2 + p^2q^2 - 3 = (pq+1)^2.
\]
\end{lem}
\begin{thm}\label{main3}
Let $p$ and $q$ be two odd prime powers with $q-p=2$, and $H_p,H_q$ as defined in Remark \ref{r1}. 
There is a Hadamard matrix of order $m=4(1+p^2)(1+q^2)$ with excess $m\sqrt{m-16}=2m(1+pq)$.
\end{thm}
\begin{proof}
Assuming $H_p$ has row sum $1-pi$ and $H_q$ has row sum $1+qi$, let $\mathbb{H}_0=H_p\otimes H_q$. The row sum of $\mathbb{H}_0$ is $pq+1+(q-p)i$. Let $\mathbb{H}_1=\mathbb{H}_0\otimes M$. Then $\mathbb{H}_1$ has row sum $pq-1+(pq+3)i$. Finally, let $\mathbb{H}$ be the Hadamard matrix of order $m=4(p^2+1)(q^2+1)$ derived from $\mathbb{H}_1$. The excess of $\mathbb{H}$ is $m\sqrt{m-16}$. By Lemma \ref{twin}, $\sqrt{m-16}=2(1+pq)$.
\end{proof}
\begin{rem}
Theorems \ref{main-excess} and \ref{main3} provide a guide on how to generate Hadamard matrices with large excesses. \end{rem}
\begin{thm}
There is a regular quaternary Hadamard matrix of order $2(1+p^2)^2$ with row sum $1+p^2 + i (1+p^2)$ leading to regular Hadamard matrices of order $4(1+p^2)^2$, for each odd prime power $p$. \end{thm}
\begin{proof}
Let $H_{j}$ be the skew-regular quaternary Hadamard matrix of order $1+p^2$ with row sum $1-pi$ for  $j=1$ and $1+pi$ for $j=2$. The regular quaternary Hadamard matrix $\mathbb{H}_0=H_1\otimes H_2$ has row sum $1+p^2$. The matrix $\mathbb{H}_1=M\otimes \mathbb{H}_0$ has row sum $1+p^2+i(1+p^2)$ and the Hadamard matrix $\mathbb{H}$ of order $4(1+p^2)^2$ derived from $\mathbb{H}_1$ has row sum $2(1+p^2)$ and is regular.
\end{proof}

\section*{Acknowledgement} Both authors are {supported by the Natural Sciences and Engineering Research Council of Canada (NSERC).}

\begin{appendices}
\newpage

\section{Example of main construction}\label{app}
Suppose $p=5$. We have \[ H = \scriptsize \setlength{\arraycolsep}{1.5pt} \left[\begin{array}{cccccccccccccccccccccccccc}
1 &j &j &j &j &j &j &j &j &j &j &j &j &j &j &j &j &j &j &j &j &j &j &j &j &j 
\\
j & 1 &j &j &j &j &j & i & i &j & i &j &j & i & i & i &j & i & i & i &j &j & i &j & i & i 
\\
j &j & 1 &j &j &j & i &j & i & i &j & i &j &j & i & i &j &j & i & i & i & i &j & i &j & i 
\\
j &j &j & 1 &j &j &j & i &j & i & i & i & i &j &j & i & i &j &j & i & i & i & i &j & i &j 
\\
j &j &j &j & 1 &j & i &j & i &j & i & i & i & i &j &j & i & i &j &j & i &j & i & i &j & i 
\\
j &j &j &j &j & 1 & i & i &j & i &j &j & i & i & i &j & i & i & i &j &j & i &j & i & i &j 
\\
j &j & i &j & i & i & 1 &j &j &j &j &j & i & i &j & i &j &j & i & i & i &j & i & i & i &j 
\\
j & i &j & i &j & i &j & 1 &j &j &j & i &j & i & i &j & i &j &j & i & i &j &j & i & i & i 
\\
j & i & i &j & i &j &j &j & 1 &j &j &j & i &j & i & i & i & i &j &j & i & i &j &j & i & i 
\\
j &j & i & i &j & i &j &j &j & 1 &j & i &j & i &j & i & i & i & i &j &j & i & i &j &j & i 
\\
j & i &j & i & i &j &j &j &j &j & 1 & i & i &j & i &j &j & i & i & i &j & i & i & i &j &j 
\\
j &j & i & i & i &j &j & i &j & i & i & 1 &j &j &j &j &j & i & i &j & i &j &j & i & i & i 
\\
j &j &j & i & i & i & i &j & i &j & i &j & 1 &j &j &j & i &j & i & i &j & i &j &j & i & i 
\\
j & i &j &j & i & i & i & i &j & i &j &j &j & 1 &j &j &j & i &j & i & i & i & i &j &j & i 
\\
j & i & i &j &j & i &j & i & i &j & i &j &j &j & 1 &j & i &j & i &j & i & i & i & i &j &j 
\\
j & i & i & i &j &j & i &j & i & i &j &j &j &j &j & 1 & i & i &j & i &j &j & i & i & i &j 
\\
j &j &j & i & i & i &j & i & i & i &j &j & i &j & i & i & 1 &j &j &j &j &j & i & i &j & i 
\\
j & i &j &j & i & i &j &j & i & i & i & i &j & i &j & i &j & 1 &j &j &j & i &j & i & i &j 
\\
j & i & i &j &j & i & i &j &j & i & i & i & i &j & i &j &j &j & 1 &j &j &j & i &j & i & i 
\\
j & i & i & i &j &j & i & i &j &j & i &j & i & i &j & i &j &j &j & 1 &j & i &j & i &j & i 
\\
j &j & i & i & i &j & i & i & i &j &j & i &j & i & i &j &j &j &j &j & 1 & i & i &j & i &j 
\\
j &j & i & i &j & i &j &j & i & i & i &j & i & i & i &j &j & i &j & i & i & 1 &j &j &j &j 
\\
j & i &j & i & i &j & i &j &j & i & i &j &j & i & i & i & i &j & i &j & i &j & 1 &j &j &j 
\\
j &j & i &j & i & i & i & i &j &j & i & i &j &j & i & i & i & i &j & i &j &j &j & 1 &j &j 
\\
j & i &j & i &j & i & i & i & i &j &j & i & i &j &j & i &j & i & i &j & i &j &j &j & 1 &j 
\\
j & i & i &j & i &j &j & i & i & i &j & i & i & i &j &j & i &j & i & i &j &j &j &j &j & 1 
\end{array}\right]
\]

and we can take $v$ to be
\[
\left[ \scriptsize
\begin{array}{*{26}{c}} 
1 & 1 & 1 & 1 & 1 & 1 & -i & -i & -i & -i & -i & -i & -i & -i & -i & -i & i & i & i & i & i & i & i & i & i & i  \\[-1ex]
\multicolumn{1}{c}{\underbrace{\hphantom{1}}_{1}}
& \multicolumn{5}{c}{\underbrace{\hphantom{\quad \quad \quad \quad \quad  \quad}}_{\mathbb{F}_p}}
& \multicolumn{5}{c}{\underbrace{\hphantom{\quad \quad \quad \quad \quad \quad \quad \quad \quad  \quad}}_{C_1}}
& \multicolumn{5}{c}{\underbrace{\hphantom{\quad \quad \quad \quad \quad \quad \quad \quad \quad  \quad}}_{C_2}}
& \multicolumn{5}{c}{\underbrace{\hphantom{\quad \quad \quad \quad \quad  \quad}}_{C_3}}
& \multicolumn{5}{c}{\underbrace{\hphantom{\quad \quad \quad \quad \quad  \quad}}_{C_4}}
\end{array}
\right].
\]
Then \[S = MH M^* = \scriptsize \setlength{\arraycolsep}{1.5pt}\left[\begin{array}{cccccccccccccccccccccccccc}
1 & j & j & j & j & j & - & - & - & - & - & - & - & - & - & - & 1 & 1 & 1 & 1 & 1 & 1 & 1 & 1 & 1 & 1 
\\
 j & 1 & j & j & j & j & - & 1 & 1 & - & 1 & - & - & 1 & 1 & 1 & 1 & - & - & - & 1 & 1 & - & 1 & - & - 
\\
 j & j & 1 & j & j & j & 1 & - & 1 & 1 & - & 1 & - & - & 1 & 1 & 1 & 1 & - & - & - & - & 1 & - & 1 & - 
\\
 j & j & j & 1 & j & j & - & 1 & - & 1 & 1 & 1 & 1 & - & - & 1 & - & 1 & 1 & - & - & - & - & 1 & - & 1 
\\
 j & j & j & j & 1 & j & 1 & - & 1 & - & 1 & 1 & 1 & 1 & - & - & - & - & 1 & 1 & - & 1 & - & - & 1 & - 
\\
 j & j & j & j & j & 1 & 1 & 1 & - & 1 & - & - & 1 & 1 & 1 & - & - & - & - & 1 & 1 & - & 1 & - & - & 1 
\\
 1 & 1 & - & 1 & - & - & 1 & j & j & j & j & j & i & i & j & i & i & i & j & j & j & i & j & j & j & i 
\\
 1 & - & 1 & - & 1 & - & j & 1 & j & j & j & i & j & i & i & j & j & i & i & j & j & i & i & j & j & j 
\\
 1 & - & - & 1 & - & 1 & j & j & 1 & j & j & j & i & j & i & i & j & j & i & i & j & j & i & i & j & j 
\\
 1 & 1 & - & - & 1 & - & j & j & j & 1 & j & i & j & i & j & i & j & j & j & i & i & j & j & i & i & j 
\\
 1 & - & 1 & - & - & 1 & j & j & j & j & 1 & i & i & j & i & j & i & j & j & j & i & j & j & j & i & i 
\\
 1 & 1 & - & - & - & 1 & j & i & j & i & i & 1 & j & j & j & j & i & j & j & i & j & i & i & j & j & j 
\\
 1 & 1 & 1 & - & - & - & i & j & i & j & i & j & 1 & j & j & j & j & i & j & j & i & j & i & i & j & j 
\\
 1 & - & 1 & 1 & - & - & i & i & j & i & j & j & j & 1 & j & j & i & j & i & j & j & j & j & i & i & j 
\\
 1 & - & - & 1 & 1 & - & j & i & i & j & i & j & j & j & 1 & j & j & i & j & i & j & j & j & j & i & i 
\\
 1 & - & - & - & 1 & 1 & i & j & i & i & j & j & j & j & j & 1 & j & j & i & j & i & i & j & j & j & i 
\\
 - & - & - & 1 & 1 & 1 & i & j & j & j & i & i & j & i & j & j & 1 & j & j & j & j & j & i & i & j & i 
\\
 - & 1 & - & - & 1 & 1 & i & i & j & j & j & j & i & j & i & j & j & 1 & j & j & j & i & j & i & i & j 
\\
 - & 1 & 1 & - & - & 1 & j & i & i & j & j & j & j & i & j & i & j & j & 1 & j & j & j & i & j & i & i 
\\
 - & 1 & 1 & 1 & - & - & j & j & i & i & j & i & j & j & i & j & j & j & j & 1 & j & i & j & i & j & i 
\\
 - & - & 1 & 1 & 1 & - & j & j & j & i & i & j & i & j & j & i & j & j & j & j & 1 & i & i & j & i & j 
\\
 - & - & 1 & 1 & - & 1 & i & i & j & j & j & i & j & j & j & i & j & i & j & i & i & 1 & j & j & j & j 
\\
 - & 1 & - & 1 & 1 & - & j & i & i & j & j & i & i & j & j & j & i & j & i & j & i & j & 1 & j & j & j 
\\
 - & - & 1 & - & 1 & 1 & j & j & i & i & j & j & i & i & j & j & i & i & j & i & j & j & j & 1 & j & j 
\\
 - & 1 & - & 1 & - & 1 & j & j & j & i & i & j & j & i & i & j & j & i & i & j & i & j & j & j & 1 & j 
\\
 - & 1 & 1 & - & 1 & - & i & j & j & j & i & j & j & j & i & i & i & j & i & i & j & j & j & j & j & 1 
\end{array}\right]
\]

\section{Skew-regular quaternary Hadamard of order 50 and equivalence to row sums $\pm 5 \pm 5i$.}\label{app2}
Below is the skew-regular quaternary Hadamard with row sums $1-7i$ coming from Theorem \ref{main}.
\[H=\scriptsize \setlength{\arraycolsep}{0.5pt}\left[\begin{array}{cccccccccccccccccccccccccccccccccccccccccccccccccc}
1 & j & j & j & j & j & j & j & 1 & 1 & 1 & 1 & 1 & 1 & 1 & 1 & 1 & 1 & 1 & 1 & 1 & 1 & 1 & 1 & 1 & 1 & 1 & 1 & 1 & - & - & - & - & - & - & - & - & - & - & - & - & - & - & - & - & - & - & - & - & -
\\
 j & 1 & j & j & j & j & j & j & - & 1 & - & - & 1 & - & 1 & - & 1 & 1 & - & - & 1 & - & - & - & - & 1 & 1 & 1 & - & 1 & 1 & - & - & - & 1 & 1 & 1 & 1 & - & 1 & 1 & - & - & 1 & - & 1 & - & 1 & 1 & -
\\
 j & j & 1 & j & j & j & j & j & 1 & - & 1 & - & - & 1 & - & - & - & 1 & 1 & - & - & 1 & - & - & - & - & 1 & 1 & 1 & 1 & 1 & 1 & - & - & - & 1 & - & 1 & 1 & - & 1 & 1 & - & - & 1 & - & 1 & - & 1 & 1
\\
 j & j & j & 1 & j & j & j & j & - & 1 & - & 1 & - & - & 1 & 1 & - & - & 1 & 1 & - & - & 1 & - & - & - & - & 1 & 1 & 1 & 1 & 1 & 1 & - & - & - & - & - & 1 & 1 & - & 1 & 1 & 1 & - & 1 & - & 1 & - & 1
\\
 j & j & j & j & 1 & j & j & j & 1 & - & 1 & - & 1 & - & - & - & 1 & - & - & 1 & 1 & - & 1 & 1 & - & - & - & - & 1 & - & 1 & 1 & 1 & 1 & - & - & 1 & - & - & 1 & 1 & - & 1 & 1 & 1 & - & 1 & - & 1 & -
\\
 j & j & j & j & j & 1 & j & j & - & 1 & - & 1 & - & 1 & - & - & - & 1 & - & - & 1 & 1 & 1 & 1 & 1 & - & - & - & - & - & - & 1 & 1 & 1 & 1 & - & 1 & 1 & - & - & 1 & 1 & - & - & 1 & 1 & - & 1 & - & 1
\\
 j & j & j & j & j & j & 1 & j & - & - & 1 & - & 1 & - & 1 & 1 & - & - & 1 & - & - & 1 & - & 1 & 1 & 1 & - & - & - & - & - & - & 1 & 1 & 1 & 1 & - & 1 & 1 & - & - & 1 & 1 & 1 & - & 1 & 1 & - & 1 & -
\\
 j & j & j & j & j & j & j & 1 & 1 & - & - & 1 & - & 1 & - & 1 & 1 & - & - & 1 & - & - & - & - & 1 & 1 & 1 & - & - & 1 & - & - & - & 1 & 1 & 1 & 1 & - & 1 & 1 & - & - & 1 & - & 1 & - & 1 & 1 & - & 1
\\
 - & 1 & - & 1 & - & 1 & 1 & - & 1 & j & j & j & j & j & j & i & j & i & i & j & i & j & i & j & j & i & i & j & i & j & j & j & i & i & i & j & j & j & i & i & i & j & j & j & j & i & j & j & i & i
\\
 - & - & 1 & - & 1 & - & 1 & 1 & j & 1 & j & j & j & j & j & j & i & j & i & i & j & i & i & i & j & j & i & i & j & j & j & j & j & i & i & i & j & j & j & i & i & i & j & i & j & j & i & j & j & i
\\
 - & 1 & - & 1 & - & 1 & - & 1 & j & j & 1 & j & j & j & j & i & j & i & j & i & i & j & j & i & i & j & j & i & i & i & j & j & j & j & i & i & j & j & j & j & i & i & i & i & i & j & j & i & j & j
\\
 - & 1 & 1 & - & 1 & - & 1 & - & j & j & j & 1 & j & j & j & j & i & j & i & j & i & i & i & j & i & i & j & j & i & i & i & j & j & j & j & i & i & j & j & j & j & i & i & j & i & i & j & j & i & j
\\
 - & - & 1 & 1 & - & 1 & - & 1 & j & j & j & j & 1 & j & j & i & j & i & j & i & j & i & i & i & j & i & i & j & j & i & i & i & j & j & j & j & i & i & j & j & j & j & i & j & j & i & i & j & j & i
\\
 - & 1 & - & 1 & 1 & - & 1 & - & j & j & j & j & j & 1 & j & i & i & j & i & j & i & j & j & i & i & j & i & i & j & j & i & i & i & j & j & j & i & i & i & j & j & j & j & i & j & j & i & i & j & j
\\
 - & - & 1 & - & 1 & 1 & - & 1 & j & j & j & j & j & j & 1 & j & i & i & j & i & j & i & j & j & i & i & j & i & i & j & j & i & i & i & j & j & j & i & i & i & j & j & j & j & i & j & j & i & i & j
\\
 - & 1 & 1 & - & 1 & 1 & - & - & i & j & i & j & i & i & j & 1 & j & j & j & j & j & j & i & j & i & i & j & i & j & j & i & i & j & j & i & j & j & j & j & i & i & i & j & j & j & i & i & i & j & j
\\
 - & - & 1 & 1 & - & 1 & 1 & - & j & i & j & i & j & i & i & j & 1 & j & j & j & j & j & j & i & j & i & i & j & i & j & j & i & i & j & j & i & j & j & j & j & i & i & i & j & j & j & i & i & i & j
\\
 - & - & - & 1 & 1 & - & 1 & 1 & i & j & i & j & i & j & i & j & j & 1 & j & j & j & j & i & j & i & j & i & i & j & i & j & j & i & i & j & j & i & j & j & j & j & i & i & j & j & j & j & i & i & i
\\
 - & 1 & - & - & 1 & 1 & - & 1 & i & i & j & i & j & i & j & j & j & j & 1 & j & j & j & j & i & j & i & j & i & i & j & i & j & j & i & i & j & i & i & j & j & j & j & i & i & j & j & j & j & i & i
\\
 - & 1 & 1 & - & - & 1 & 1 & - & j & i & i & j & i & j & i & j & j & j & j & 1 & j & j & i & j & i & j & i & j & i & j & j & i & j & j & i & i & i & i & i & j & j & j & j & i & i & j & j & j & j & i
\\
 - & - & 1 & 1 & - & - & 1 & 1 & i & j & i & i & j & i & j & j & j & j & j & j & 1 & j & i & i & j & i & j & i & j & i & j & j & i & j & j & i & j & i & i & i & j & j & j & i & i & i & j & j & j & j
\\
 - & 1 & - & 1 & 1 & - & - & 1 & j & i & j & i & i & j & i & j & j & j & j & j & j & 1 & j & i & i & j & i & j & i & i & i & j & j & i & j & j & j & j & i & i & i & j & j & j & i & i & i & j & j & j
\\
 - & 1 & 1 & - & - & - & 1 & 1 & i & i & j & i & i & j & j & i & j & i & j & i & i & j & 1 & j & j & j & j & j & j & j & i & j & j & i & j & i & j & i & i & j & j & i & j & j & j & j & i & i & i & j
\\
 - & 1 & 1 & 1 & - & - & - & 1 & j & i & i & j & i & i & j & j & i & j & i & j & i & i & j & 1 & j & j & j & j & j & i & j & i & j & j & i & j & j & j & i & i & j & j & i & j & j & j & j & i & i & i
\\
 - & 1 & 1 & 1 & 1 & - & - & - & j & j & i & i & j & i & i & i & j & i & j & i & j & i & j & j & 1 & j & j & j & j & j & i & j & i & j & j & i & i & j & j & i & i & j & j & i & j & j & j & j & i & i
\\
 - & - & 1 & 1 & 1 & 1 & - & - & i & j & j & i & i & j & i & i & i & j & i & j & i & j & j & j & j & 1 & j & j & j & i & j & i & j & i & j & j & j & i & j & j & i & i & j & i & i & j & j & j & j & i
\\
 - & - & - & 1 & 1 & 1 & 1 & - & i & i & j & j & i & i & j & j & i & i & j & i & j & i & j & j & j & j & 1 & j & j & j & i & j & i & j & i & j & j & j & i & j & j & i & i & i & i & i & j & j & j & j
\\
 - & - & - & - & 1 & 1 & 1 & 1 & j & i & i & j & j & i & i & i & j & i & i & j & i & j & j & j & j & j & j & 1 & j & j & j & i & j & i & j & i & i & j & j & i & j & j & i & j & i & i & i & j & j & j
\\
 - & 1 & - & - & - & 1 & 1 & 1 & i & j & i & i & j & j & i & j & i & j & i & i & j & i & j & j & j & j & j & j & 1 & i & j & j & i & j & i & j & i & i & j & j & i & j & j & j & j & i & i & i & j & j
\\
 1 & - & - & - & 1 & 1 & 1 & - & j & j & i & i & i & j & j & j & j & i & j & j & i & i & j & i & j & i & j & j & i & 1 & j & j & j & j & j & j & i & j & i & i & j & i & j & i & j & j & i & i & j & i
\\
 1 & - & - & - & - & 1 & 1 & 1 & j & j & j & i & i & i & j & i & j & j & i & j & j & i & i & j & i & j & i & j & j & j & 1 & j & j & j & j & j & j & i & j & i & i & j & i & i & i & j & j & i & i & j
\\
 1 & 1 & - & - & - & - & 1 & 1 & j & j & j & j & i & i & i & i & i & j & j & i & j & j & j & i & j & i & j & i & j & j & j & 1 & j & j & j & j & i & j & i & j & i & i & j & j & i & i & j & j & i & i
\\
 1 & 1 & 1 & - & - & - & - & 1 & i & j & j & j & j & i & i & j & i & i & j & j & i & j & j & j & i & j & i & j & i & j & j & j & 1 & j & j & j & j & i & j & i & j & i & i & i & j & i & i & j & j & i
\\
 1 & 1 & 1 & 1 & - & - & - & - & i & i & j & j & j & j & i & j & j & i & i & j & j & i & i & j & j & i & j & i & j & j & j & j & j & 1 & j & j & i & j & i & j & i & j & i & i & i & j & i & i & j & j
\\
 1 & - & 1 & 1 & 1 & - & - & - & i & i & i & j & j & j & j & i & j & j & i & i & j & j & j & i & j & j & i & j & i & j & j & j & j & j & 1 & j & i & i & j & i & j & i & j & j & i & i & j & i & i & j
\\
 1 & - & - & 1 & 1 & 1 & - & - & j & i & i & i & j & j & j & j & i & j & j & i & i & j & i & j & i & j & j & i & j & j & j & j & j & j & j & 1 & j & i & i & j & i & j & i & j & j & i & i & j & i & i
\\
 1 & - & 1 & 1 & - & - & 1 & - & j & j & j & i & i & i & j & j & j & i & i & i & j & j & j & j & i & j & j & i & i & i & j & i & j & i & i & j & 1 & j & j & j & j & j & j & i & j & i & i & j & i & j
\\
 1 & - & - & 1 & 1 & - & - & 1 & j & j & j & j & i & i & i & j & j & j & i & i & i & j & i & j & j & i & j & j & i & j & i & j & i & j & i & i & j & 1 & j & j & j & j & j & j & i & j & i & i & j & i
\\
 1 & 1 & - & - & 1 & 1 & - & - & i & j & j & j & j & i & i & j & j & j & j & i & i & i & i & i & j & j & i & j & j & i & j & i & j & i & j & i & j & j & 1 & j & j & j & j & i & j & i & j & i & i & j
\\
 1 & - & 1 & - & - & 1 & 1 & - & i & i & j & j & j & j & i & i & j & j & j & j & i & i & j & i & i & j & j & i & j & i & i & j & i & j & i & j & j & j & j & 1 & j & j & j & j & i & j & i & j & i & i
\\
 1 & - & - & 1 & - & - & 1 & 1 & i & i & i & j & j & j & j & i & i & j & j & j & j & i & j & j & i & i & j & j & i & j & i & i & j & i & j & i & j & j & j & j & 1 & j & j & i & j & i & j & i & j & i
\\
 1 & 1 & - & - & 1 & - & - & 1 & j & i & i & i & j & j & j & i & i & i & j & j & j & j & i & j & j & i & i & j & j & i & j & i & i & j & i & j & j & j & j & j & j & 1 & j & i & i & j & i & j & i & j
\\
 1 & 1 & 1 & - & - & 1 & - & - & j & j & i & i & i & j & j & j & i & i & i & j & j & j & j & i & j & j & i & i & j & j & i & j & i & i & j & i & j & j & j & j & j & j & 1 & j & i & i & j & i & j & i
\\
 1 & - & 1 & - & - & 1 & - & 1 & j & i & i & j & j & i & j & j & j & j & i & i & i & j & j & j & i & i & i & j & j & i & i & j & i & i & j & j & i & j & i & j & i & i & j & 1 & j & j & j & j & j & j
\\
 1 & 1 & - & 1 & - & - & 1 & - & j & j & i & i & j & j & i & j & j & j & j & i & i & i & j & j & j & i & i & i & j & j & i & i & j & i & i & j & j & i & j & i & j & i & i & j & 1 & j & j & j & j & j
\\
 1 & - & 1 & - & 1 & - & - & 1 & i & j & j & i & i & j & j & i & j & j & j & j & i & i & j & j & j & j & i & i & i & j & j & i & i & j & i & i & i & j & i & j & i & j & i & j & j & 1 & j & j & j & j
\\
 1 & 1 & - & 1 & - & 1 & - & - & j & i & j & j & i & i & j & i & i & j & j & j & j & i & i & j & j & j & j & i & i & i & j & j & i & i & j & i & i & i & j & i & j & i & j & j & j & j & 1 & j & j & j
\\
 1 & - & 1 & - & 1 & - & 1 & - & j & j & i & j & j & i & i & i & i & i & j & j & j & j & i & i & j & j & j & j & i & i & i & j & j & i & i & j & j & i & i & j & i & j & i & j & j & j & j & 1 & j & j
\\
 1 & - & - & 1 & - & 1 & - & 1 & i & j & j & i & j & j & i & j & i & i & i & j & j & j & i & i & i & j & j & j & j & j & i & i & j & j & i & i & i & j & i & i & j & i & j & j & j & j & j & j & 1 & j
\\
 1 & 1 & - & - & 1 & - & 1 & - & i & i & j & j & i & j & j & j & j & i & i & i & j & j & j & i & i & i & j & j & j & i & j & i & i & j & j & i & j & i & j & i & i & j & i & j & j & j & j & j & j & 1
\end{array}\right]
\]
Then let 
\[v= \scriptsize \setlength{\arraycolsep}{1.5pt} \left[\begin{array}{cccccccccccccccccccccccccccccccccccccccccccccccccc}
i & 1 & i & i & j & 1 & 1 & 1 & i & i & i & i & 1 & - & - & 1 & - & i & 1 & 1 & j & - & 1 & - & - & i & 1 & 1 & j & - & i & i & - & - & i & i & - & i & - & - & j & 1 & 1 & - & 1 & i & - & - & j & 1
\end{array}\right].
\]
From $H$ and $D=\text{diag}(v)$, we obtain the following skew-semi-regular quaternary Hadamard matrix with row sums $\pm 5 \pm 5i$.
\[D H D^*=\scriptsize \setlength{\arraycolsep}{0.5pt}\left[\begin{array}{cccccccccccccccccccccccccccccccccccccccccccccccccc}
1 & 1 & j & j & i & 1 & 1 & 1 & 1 & 1 & 1 & 1 & i & j & j & i & j & 1 & i & i & - & j & i & j & j & 1 & i & i & - & i & - & - & i & i & - & - & i & - & i & i & 1 & j & j & i & j & - & i & i & 1 & j
\\
 - & 1 & - & - & 1 & j & j & j & i & j & i & i & 1 & 1 & - & - & - & j & - & - & i & 1 & - & 1 & 1 & j & 1 & 1 & j & - & j & i & 1 & 1 & j & j & - & j & 1 & - & i & - & - & - & - & j & 1 & - & i & -
\\
 j & 1 & 1 & j & i & 1 & 1 & 1 & 1 & - & 1 & - & j & j & i & j & i & 1 & i & j & 1 & j & j & i & i & - & i & i & - & j & 1 & 1 & i & i & - & 1 & i & 1 & j & i & - & i & j & i & i & - & j & i & - & i
\\
 j & 1 & j & 1 & i & 1 & 1 & 1 & - & 1 & - & 1 & j & i & j & i & i & - & i & i & 1 & i & i & i & i & - & j & i & - & j & 1 & 1 & j & i & - & - & i & - & j & j & 1 & i & i & j & j & 1 & i & j & 1 & i
\\
 i & - & i & i & 1 & - & - & - & - & 1 & - & 1 & j & j & j & i & i & 1 & i & j & 1 & j & j & i & j & 1 & i & i & 1 & j & - & - & i & i & 1 & 1 & i & 1 & j & i & 1 & i & j & i & j & 1 & i & j & 1 & i
\\
 - & j & - & - & 1 & 1 & j & j & i & j & i & j & - & - & 1 & - & 1 & j & - & - & i & - & 1 & - & - & i & - & - & j & 1 & i & j & - & - & j & i & - & j & 1 & 1 & i & 1 & - & 1 & 1 & j & 1 & - & j & 1
\\
 - & j & - & - & 1 & j & 1 & j & i & i & j & i & 1 & 1 & - & 1 & 1 & i & 1 & - & j & - & - & - & - & j & - & - & j & 1 & i & i & - & - & j & j & 1 & j & - & 1 & j & 1 & 1 & - & - & j & - & 1 & i & -
\\
 - & j & - & - & 1 & j & j & 1 & j & i & i & j & - & - & 1 & 1 & - & i & - & 1 & j & 1 & - & 1 & - & j & 1 & - & j & - & i & i & 1 & - & j & j & - & i & - & - & j & - & 1 & 1 & 1 & i & - & - & j & 1
\\
 - & i & - & 1 & 1 & i & i & j & 1 & j & j & j & 1 & - & - & - & - & i & - & 1 & j & - & - & - & - & i & - & 1 & j & - & j & j & 1 & 1 & i & j & - & j & 1 & 1 & j & 1 & 1 & - & 1 & i & - & - & j & -
\\
 - & j & 1 & - & - & j & i & i & j & 1 & j & j & 1 & - & - & 1 & 1 & j & - & - & i & 1 & - & 1 & - & j & - & - & i & - & j & j & - & 1 & i & i & - & j & - & 1 & j & - & 1 & 1 & 1 & j & 1 & - & i & -
\\
 - & i & - & 1 & 1 & i & j & i & j & j & 1 & j & 1 & - & - & - & - & i & 1 & - & j & - & 1 & 1 & 1 & j & 1 & - & j & 1 & j & j & - & - & i & i & - & j & - & - & j & - & - & 1 & - & j & - & 1 & i & 1
\\
 - & i & 1 & - & - & j & i & j & j & j & j & 1 & 1 & - & - & 1 & 1 & j & - & 1 & j & 1 & - & - & 1 & i & 1 & 1 & j & 1 & i & j & - & - & j & i & 1 & j & - & - & i & - & - & - & - & i & - & - & j & 1
\\
 i & - & j & j & j & 1 & - & 1 & - & - & - & - & 1 & i & i & i & i & 1 & j & i & 1 & j & i & j & i & 1 & i & j & 1 & j & 1 & 1 & i & i & - & - & j & 1 & i & i & 1 & j & i & i & j & 1 & j & i & 1 & i
\\
 j & - & j & i & j & 1 & - & 1 & 1 & 1 & 1 & 1 & i & 1 & j & j & i & 1 & j & i & 1 & j & i & i & i & 1 & j & j & - & j & - & - & i & j & 1 & 1 & i & - & i & j & - & i & i & i & i & 1 & i & i & - & i
\\
 j & 1 & i & j & j & - & 1 & - & 1 & 1 & 1 & 1 & i & j & 1 & i & i & - & i & j & - & i & i & j & i & - & i & j & 1 & j & 1 & - & i & i & 1 & 1 & j & - & i & i & - & i & i & j & j & 1 & j & i & 1 & i
\\
 i & 1 & j & i & i & 1 & - & - & 1 & - & 1 & - & i & j & i & 1 & i & - & j & j & 1 & i & i & i & j & 1 & j & i & 1 & i & 1 & 1 & i & i & 1 & - & i & - & i & j & - & i & j & i & j & 1 & j & j & 1 & j
\\
 j & 1 & i & i & i & - & - & 1 & 1 & - & 1 & - & i & i & i & i & 1 & 1 & i & i & - & j & i & i & j & - & j & i & 1 & j & 1 & - & i & j & 1 & - & j & 1 & j & j & 1 & j & j & j & i & 1 & i & i & 1 & i
\\
 - & j & - & 1 & - & j & i & i & i & j & i & j & - & - & 1 & 1 & - & 1 & 1 & 1 & i & - & - & - & 1 & j & - & - & i & 1 & j & j & 1 & 1 & j & j & 1 & j & - & - & i & - & - & - & 1 & j & - & 1 & j & -
\\
 i & 1 & i & i & i & 1 & - & 1 & 1 & 1 & - & 1 & j & j & i & j & i & - & 1 & j & 1 & i & j & j & i & 1 & j & i & - & i & 1 & - & i & j & 1 & - & j & 1 & i & i & 1 & j & i & j & j & - & i & i & - & i
\\
 i & 1 & j & i & j & 1 & 1 & - & - & 1 & 1 & - & i & i & j & j & i & - & j & 1 & 1 & i & i & i & j & - & i & j & - & i & - & 1 & i & i & 1 & 1 & j & 1 & j & i & 1 & j & j & j & i & - & i & i & 1 & i
\\
 1 & i & - & - & - & i & j & j & j & i & j & j & - & - & 1 & - & 1 & i & - & - & 1 & 1 & 1 & - & 1 & j & - & 1 & j & - & i & i & - & 1 & i & j & 1 & j & - & - & j & - & - & - & 1 & j & 1 & 1 & j & -
\\
 j & - & j & i & j & 1 & 1 & - & 1 & - & 1 & - & j & j & i & i & j & 1 & i & i & - & 1 & i & i & i & 1 & j & i & 1 & i & - & 1 & j & i & 1 & 1 & j & 1 & i & i & 1 & i & i & j & j & - & i & j & - & i
\\
 i & 1 & j & i & j & - & 1 & 1 & 1 & 1 & - & 1 & i & i & i & i & i & 1 & j & i & - & i & 1 & i & i & - & j & j & 1 & i & 1 & - & i & j & - & 1 & i & 1 & j & i & 1 & i & j & i & j & - & j & j & - & j
\\
 j & - & i & i & i & 1 & 1 & - & 1 & - & - & 1 & j & i & j & i & i & 1 & j & i & 1 & i & i & 1 & j & 1 & i & i & - & i & 1 & - & j & j & - & 1 & j & 1 & i & i & - & i & j & j & i & 1 & j & i & 1 & j
\\
 j & - & i & i & j & 1 & 1 & 1 & 1 & 1 & - & - & i & i & i & j & j & - & i & j & - & i & i & j & 1 & 1 & i & i & - & j & - & 1 & i & j & 1 & - & i & 1 & j & i & 1 & i & i & i & i & 1 & j & j & 1 & j
\\
 - & j & 1 & 1 & - & i & j & j & i & j & j & i & - & - & 1 & - & 1 & j & - & 1 & j & - & 1 & - & - & 1 & 1 & 1 & i & 1 & j & i & - & 1 & j & j & - & i & - & - & j & - & 1 & 1 & - & j & - & - & i & -
\\
 i & - & i & j & i & 1 & 1 & - & 1 & 1 & - & - & i & j & i & j & j & 1 & j & i & 1 & j & j & i & i & - & 1 & j & 1 & i & 1 & - & j & i & 1 & - & i & - & j & i & 1 & i & i & j & i & 1 & i & i & 1 & j
\\
 i & - & i & i & i & 1 & 1 & 1 & - & 1 & 1 & - & j & j & j & i & i & 1 & i & j & - & i & j & i & i & - & j & 1 & 1 & i & - & 1 & i & j & - & 1 & j & - & i & j & 1 & j & i & i & i & 1 & j & i & 1 & j
\\
 1 & j & 1 & 1 & - & j & j & j & j & i & j & j & - & 1 & - & - & - & i & 1 & 1 & j & - & - & 1 & 1 & i & - & - & 1 & - & i & i & - & 1 & j & i & - & j & 1 & 1 & i & - & - & 1 & - & j & - & - & j & -
\\
 i & 1 & j & j & j & - & - & 1 & 1 & 1 & - & - & j & j & j & i & j & - & i & i & 1 & i & i & i & j & - & i & i & 1 & 1 & 1 & 1 & j & j & 1 & 1 & i & 1 & i & i & - & j & i & i & i & 1 & i & i & - & j
\\
 1 & j & - & - & 1 & i & i & i & j & j & j & i & - & 1 & - & - & - & j & - & 1 & i & 1 & - & - & 1 & j & - & 1 & i & - & 1 & j & - & - & j & j & - & i & - & 1 & j & 1 & - & 1 & - & j & - & 1 & j & 1
\\
 1 & i & - & - & 1 & j & i & i & j & j & j & j & - & 1 & 1 & - & 1 & j & 1 & - & i & - & 1 & 1 & - & i & 1 & - & i & - & j & 1 & - & - & j & j & 1 & j & 1 & - & j & - & 1 & - & - & i & - & - & j & -
\\
 i & - & i & j & i & 1 & 1 & - & - & 1 & 1 & 1 & i & i & i & i & i & - & i & i & 1 & j & i & j & i & 1 & j & i & 1 & j & 1 & 1 & 1 & j & 1 & 1 & j & - & j & i & - & j & j & i & i & - & i & j & - & j
\\
 i & - & i & i & i & 1 & 1 & 1 & - & - & 1 & 1 & i & j & i & i & j & - & j & i & - & i & j & j & j & - & i & j & - & j & 1 & 1 & j & 1 & 1 & 1 & i & 1 & i & j & 1 & i & j & i & j & 1 & i & i & - & i
\\
 1 & j & 1 & 1 & - & j & j & j & i & i & i & j & 1 & - & - & - & - & j & - & - & i & - & 1 & 1 & - & j & - & 1 & j & - & j & j & - & - & 1 & j & 1 & i & - & 1 & i & - & 1 & - & - & i & - & 1 & j & 1
\\
 1 & j & - & 1 & - & i & j & j & j & i & i & i & 1 & - & - & 1 & 1 & j & 1 & - & j & - & - & - & 1 & j & 1 & - & i & - & j & j & - & - & j & 1 & - & i & 1 & - & j & 1 & - & - & 1 & i & 1 & - & j & -
\\
 i & 1 & i & i & i & 1 & - & 1 & 1 & 1 & 1 & - & j & i & j & i & j & - & j & j & - & j & i & j & i & 1 & i & j & 1 & i & 1 & - & j & i & - & 1 & 1 & 1 & j & j & - & i & i & i & i & - & i & j & 1 & i
\\
 1 & j & - & 1 & - & j & j & i & j & j & j & j & - & 1 & 1 & 1 & - & j & - & - & j & - & - & - & - & i & 1 & 1 & j & - & i & j & 1 & - & i & i & - & 1 & - & - & i & 1 & 1 & - & - & j & 1 & 1 & i & -
\\
 i & - & j & j & j & - & 1 & 1 & - & 1 & 1 & 1 & i & i & i & i & j & 1 & i & j & 1 & i & j & i & j & 1 & j & i & - & i & 1 & - & j & i & 1 & - & j & 1 & 1 & j & - & i & i & i & i & - & j & i & 1 & i
\\
 i & 1 & i & j & i & - & - & 1 & - & - & 1 & 1 & i & j & i & j & j & 1 & i & i & 1 & i & i & i & i & 1 & i & j & - & i & - & 1 & i & j & - & 1 & j & 1 & j & 1 & - & i & i & j & j & 1 & i & j & 1 & j
\\
 - & i & 1 & - & - & i & j & j & j & j & j & i & - & 1 & 1 & 1 & - & i & - & - & j & - & - & 1 & - & j & - & - & i & 1 & j & j & 1 & - & i & j & 1 & i & 1 & 1 & 1 & - & - & - & - & j & 1 & - & j & 1
\\
 j & 1 & i & i & i & - & - & 1 & - & 1 & 1 & 1 & j & i & i & i & j & 1 & j & j & 1 & i & i & i & i & 1 & i & j & 1 & j & - & 1 & j & i & 1 & - & i & - & i & i & 1 & 1 & j & j & i & - & j & i & - & j
\\
 j & 1 & j & i & j & 1 & - & - & - & - & 1 & 1 & i & i & i & j & j & 1 & i & j & 1 & i & j & j & i & - & i & i & 1 & i & 1 & - & j & j & - & 1 & i & - & i & i & 1 & j & 1 & i & i & 1 & i & j & 1 & i
\\
 i & 1 & i & j & i & - & 1 & - & 1 & - & - & 1 & i & i & j & i & j & 1 & j & j & 1 & j & i & j & i & - & j & i & - & i & - & 1 & i & i & 1 & 1 & i & 1 & i & j & 1 & j & i & 1 & i & 1 & j & j & - & i
\\
 j & 1 & i & j & j & - & 1 & - & - & - & 1 & 1 & j & i & j & j & i & - & j & i & - & j & j & i & i & 1 & i & i & 1 & i & 1 & 1 & i & j & 1 & - & i & 1 & i & j & 1 & i & i & i & 1 & - & i & i & 1 & j
\\
 1 & j & 1 & - & - & j & j & i & i & j & j & i & - & - & - & - & - & j & 1 & 1 & j & 1 & 1 & - & - & j & - & - & j & - & j & i & 1 & - & i & i & 1 & j & 1 & - & j & 1 & - & - & 1 & 1 & - & - & i & 1
\\
 i & - & j & i & i & - & 1 & 1 & 1 & - & 1 & 1 & j & i & j & j & i & 1 & i & i & - & i & j & j & j & 1 & i & j & 1 & i & 1 & 1 & i & i & 1 & - & i & - & j & i & - & j & i & j & i & 1 & 1 & j & - & i
\\
 i & 1 & i & j & j & 1 & - & 1 & 1 & 1 & - & 1 & i & i & i & j & i & - & i & i & - & j & j & i & j & 1 & i & i & 1 & i & - & 1 & j & i & - & 1 & j & - & i & j & 1 & i & j & j & i & 1 & j & 1 & - & i
\\
 - & i & 1 & - & - & j & i & j & j & i & i & j & - & 1 & - & - & - & j & 1 & - & j & 1 & 1 & - & - & i & - & - & j & 1 & j & j & 1 & 1 & j & j & - & i & - & - & j & 1 & - & 1 & - & i & 1 & 1 & 1 & -
\\
 j & 1 & i & i & i & - & 1 & - & 1 & 1 & - & - & i & i & i & j & i & 1 & i & i & 1 & i & j & j & j & 1 & j & j & 1 & j & - & 1 & j & i & - & 1 & i & 1 & i & j & - & j & i & i & j & - & i & i & 1 & 1
\end{array}\right]
\]

\end{appendices}
\end{document}